\documentclass{amsart}
\usepackage{amsmath}
\usepackage{latexsym,amsxtra,amscd,ifthen}
\usepackage{amsfonts}
\usepackage{color,graphicx} 
\usepackage{verbatim}
\usepackage{amsmath}
\usepackage{amsthm}
\usepackage{amssymb}
\usepackage{url}

\theoremstyle{plain}

\def\MSA/{\mbox{\textsc{ms~access}}}
\newtheorem{theorem}{Theorem}
\newtheorem{lemma}[theorem]{Lemma}
\newtheorem{proposition}[theorem]{Proposition}
\newtheorem{corollary}[theorem]{Corollary}
\newtheorem{conjecture}[theorem]{Conjecture}

\numberwithin{theorem}{section}
\numberwithin{equation}{theorem}

\theoremstyle{definition}
\newtheorem{definition}[theorem]{Definition}

\newtheorem{remark}[theorem]{Remark}
\newtheorem{question}[theorem]{Question}
\newtheorem*{question*}{Question}

\DeclareMathOperator{\LND}{LND}
\DeclareMathOperator{\ML}{ML}
\DeclareMathOperator{\Der}{Der}
\DeclareMathOperator{\Ker}{{\rm ker}}

\begin{document}

\title[Cancellation in noncommutative algebras]{Noncommutative analogues of a cancellation theorem of Abhyankar, Eakin, and Heinzer}

\author{Jason Bell}\thanks{The first-named author was supported by NSERC grant RGPIN-2016-03632.}

\address{Bell: Department of Pure Mathematics, University of Waterloo, Canada}

\email{jpbell@uwaterloo.ca}

\author{Maryam Hamidizadeh}

\address{Hamidizadeh: Faculty of Mathematical Sciences, Shahrood University of Technology, Iran}

\email{maryamhamidizadeh@yahoo.com}

\author{Hongdi Huang}

\address{Huang: Department of Pure Mathematics, University of Waterloo, Canada}

\email{h237huan@uwaterloo.ca}

\author{Helbert Venegas}

\address{Venegas: Departamento de Matem{\'a}ticas, Universidad Nacional de Colombia, Sede Bogot{\'a}, Colombia}

\email{hjvenegasr@unal.edu.co}

\begin{abstract}
Let $k$ be a field and let $A$ be a finitely generated $k$-algebra.  The algebra $A$ is said to be cancellative if whenever $B$ is another $k$-algebra with the property that $A[x]\cong B[x]$ then we necessarily have $A\cong B$.  An important result of Abhyankar, Eakin, and Heinzer shows that if $A$ is a finitely generated commutative integral domain of Krull dimension one then it is cancellative.  We consider the question of cancellation for finitely generated not-necessarily-commutative domains of Gelfand-Kirillov dimension one, and show that such algebras are necessarily cancellative when the characteristic of the base field is zero.  In particular, this recovers the cancellation result of Abhyankar, Eakin, and Heinzer in characteristic zero when one restricts to the commutative case. We also provide examples that show affine domains of Gelfand-Kirillov dimension one need not be cancellative when the base field has positive characteristic, giving a counterexample to a conjecture of Tang, the fourth-named author, and Zhang. In addition, we prove a skew analogue of the result of Abhyankar-Eakin-Heinzer, in which one works with skew polynomial extensions as opposed to ordinary polynomial rings.
\end{abstract}





\maketitle

\section{Introduction}
A longstanding problem in affine algebraic geometry is the Zariski cancellation problem, which asks whether an affine variety $X$ over an algebraically closed field $k$ having the property that $X\times \mathbb{A}^1\cong \mathbb{A}^{n+1}$ is necessarily isomorphic to $\mathbb{A}^n$. The question is known to have an affirmative answer when $n=1$ \cite{AEH}, and $n=2$, with the characteristic zero case being done by Fujita \cite{Fuj} and  Miyanishi-Sugie \cite{MS}, and the positive characteristic case handled by Russell \cite{Rus}. In positive characteristic, Gupta \cite{Gu1, Gu2} gave counterexamples to the Zariski cancellation problem in dimension at least three, but the problem remains open in dimension greater than two in the case that the base field has characteristic zero.  

One can ask more generally to determine which affine varieties $X$ are cancellative in the sense that when $X\times \mathbb{A}^1\cong Y\times \mathbb{A}^1$ for some affine variety $Y$ we must have $X\cong Y$.  In this general setting, there is more pathological behaviour and Danielewski \cite{Dan} and Hochster \cite{Hoc} gave examples of affine complex varieties that are not cancellative.  It should be noted, however, that these non-cancellative examples all have dimension at least two, and if one restricts one's attention to curves, cancellation holds: this is a result of Abhyankar, Eakin, and Heinzer \cite{AEH}; in fact, they prove more: they show that if $d\ge 1$ and $R$ and $S$ are the coordinate rings of affine curves, then if $R[x_1,\ldots ,x_d]\cong S[x_1,\ldots ,x_d]$ then we have $R\cong S$.  

In recent years, increased attention has been paid to the noncommutative analogue of the Zariski cancellation problem \cite{BZ1, GW, LWZ, LuWZ, TVZ}.  In this setting one has a field $k$ and a finitely 
generated $k$-algebra $R$ and one asks whether a $k$-algebra isomorphism $R[x]\cong S[x]$ implies $R\cong S$ when $S$ is another finitely generated $k$-algebra.  Algebras $R$ that have this property 
are said to be \emph{cancellative}.  An algebra $R$ is \emph{strongly cancellative} if, for every $d\ge 1$, an isomorphism $R[x_1,\ldots ,x_d]\cong S[x_1,\ldots ,x_d]$ implies that $R$ is isomorphic to $S$. It is known that many classes of noncommutative algebras are cancellative or strongly cancellative in the sense above.  Notably, cancellation holds for algebras with trivial centre, 
for ``non-commutative surfaces'' that are not commutative, and many quantizations of coordinate rings of affine varieties (see the results in \cite{BZ1}).  

The goal of this paper is to look at non-commutative 
analogues of  the result of Abhyankar, Eakin, and Heinzer.  Their theorem, when one works in the category of commutative algebras, says that if $A$ is a finitely generated algebra that is an integral domain of 
Krull dimension one, then $A$ is strongly cancellative in the above sense.  We consider a noncommutative analogue of this theorem, in which one considers finitely generated domains of Gelfand-Kirillov dimension one.  
When working with noncommutative algebras, it is generally preferable to work with Gelfand-Kirillov dimension rather than with the classical Krull dimension, and
Gelfand-Kirillov dimension and Krull dimension coincide when one restricts one's focus to the class of finitely generated commutative algebras over a field.  For more information about Gelfand-Kirillov dimension, we refer the reader to the book of Krause and 
Lenagan \cite{KL}. Throughout this paper, when $A$ is a finitely generated $k$-algebra, we will simply say that $A$ is an \emph{affine algebra over} $k$, or simply an \emph{affine algebra} when the base field is understood; we shall also let $Z(A)$ denote the centre of an algebra $A$.  Our main result is the following theorem.
\begin{theorem}
\label{thm: main}
We have the following results for affine domains of Gelfand-Kirillov dimension one. 
\begin{enumerate}
\item Let $k$ be a field of characteristic zero and let $A$ be an affine domain over $k$ of Gelfand-Kirillov dimension one. Then $A$ is cancellative.
\item Let $p$ be prime.  Then there exists a field $k$ of characteristic $p$ and an affine domain $A$ of Gelfand-Kirillov dimension one that is not cancellative.
\end{enumerate}
\end{theorem}
Part (b) of Theorem \ref{thm: main} gives a counterexample to a conjecture of \cite[Conjecture 0.3(1)]{TVZ}, while Theorem \ref{thm: main} (a) answers a question of Lezama, Wang, and Zhang \cite[Question 0.5]{LWZ} in the case when the base field has characteristic zero in the domain case.  Since Krull dimension and Gelfand-Kirillov dimension coincide for finitely generated commutative $k$-algebras, Theorem \ref{thm: main} specializes to the classical cancellation result of Abhyankar-Eakin-Heinzer in the case of characteristic zero base fields when one takes $R$ to be commutative.  We note that \cite{AEH} show in fact such rings are strongly cancellative and we do not know whether this conclusion holds in characteristic zero for Theorem \ref{thm: main} (a). We also point out that Lezama, Wang, and Zhang \cite[Theorem 0.6]{LWZ} proved that for algebraically closed base fields $k$, affine prime $k$-algebras of Gelfand-Kirillov dimension one are cancellative.  The algebraically closed property is needed, because the authors invoke Tsen's theorem at one point in their proof. Our example, shows that this application of Tsen's theorem is in some sense necessary to get their result in positive characteristic.  

 In characteristic zero, our Theorem \ref{thm: main} (a) is somewhat orthogonal to the result of \cite{LWZ}, since domains of Gelfand-Kirillov dimension one over algebraically closed fields are commutative by an application of Tsen's theorem to a result of Small and Warfield \cite{SW} and hence the only part of Theorem \ref{thm: main} (a) covered by \cite[Theorem 0.6]{LWZ} is the commutative case, which was previously known from the result of Abhyankar-Eakin-Heinzer \cite{AEH}.

We also prove a result in a different direction; namely, skew cancellativity.  To describe this extension, we recall that given a ring $R$, an automorphism $\sigma$ of $R$ and a $\sigma$-derivation $\delta:R\to R$ of $R$ (that is, $\delta$ satisfies $\delta(rs) = \sigma(r)\delta(s)+\delta(r)s$ for $r,s\in R$), one can define a \emph{skew polynomial} extension $R[x;\sigma,\delta]$, which is just $R[x]$ as an additive abelian group and with multiplication given by $x\cdot r=\sigma(r) x + \delta(r)$ for $r\in R$, where we use the same multiplication rule for elements in $R$ as before.  The two most important special cases of this construction are the skew polynomial extensions of automorphism type, where $\delta=0$; and skew polynomial extensions of derivation type, where $\sigma$ is the identity.  In the former case, where $\delta=0$, it is customary to omit $\delta$ and write $R[x;\sigma]$; and in the latter case, where $\sigma$ is the identity, it is customary to omit $\sigma$ and write $R[x;\delta]$. In light of the Zariski cancellation problem, it is then natural to ask when an algebra $R$ is \emph{skew cancellative}; that is, if $R[x;\sigma,\delta]\cong S[x;\sigma',\delta']$ when do we necessarily have $R\cong S$? We show that this holds in the two cases just mentioned when the coefficient ring $R$ is an affine commutative domain of Krull dimension one.
\begin{theorem}\label{thm: main2}
Let $k$ be a field, let $A$ and $B$ be affine commutative integral domains of Krull dimension one, and let $\sigma,\sigma'$ be $k$-algebra automorphisms of $A$ and $B$ respectively and let $\delta, \delta'$ be $k$-linear derivations of $A$ and $B$ respectively.  If $A[x;\sigma]\cong B[x';\sigma']$ then $A\cong B$. If, in addition, $k$ has characteristic zero and if $A[x;\delta]\cong B[x';\delta']$ then $A\cong B$.
\end{theorem}
A special case of Theorem \ref{thm: main2} was proved by Bergen \cite{Ber} in the derivation case.  Specifically, he proved that if $k$ is a field of characteristic zero and $R[t;\delta]$ is isomorphic to $k[x][y;\delta']$, with $\delta'(x)\in k[x]$ having degree at least one, then $R\cong k[x]$.  It would be interesting to give a ``unification'' of the two results occurring in Theorem \ref{thm: main2} and prove that skew cancellation holds for general skew polynomial extensions, although this appears to be considerably more subtle than the cases we consider.  The positive characteristic case for skew extensions of derivation type appears to have additional subtleties.  In particular, the constructions given in \S \ref{S4} show that cancellation can behave strangely with skew extensions of derivation type in positive characteristic.

The outline of this paper is as follows.  In Section \ref{S2}, we recall the definition of the Makar-Limanov invariant and other concepts related to cancellation.  In addition, we prove a general result that suggests over ``nice'' base fields that cancellation should be controlled by the centre (see Proposition \ref{that}, Corollary \ref{corn}, and Conjecture \ref{conic}).  In Section \ref{S3}, we prove Theorem \ref{thm: main} (a) and prove some positive results for domains of Gelfand-Kirillov dimension one over positive characteristic base fields.  In Section \ref{S4}, we construct the family of examples needed to establish Theorem \ref{thm: main} (b).  Finally, in Section \ref{S5}, we consider skew cancellation and prove Theorem \ref{thm: main2}.  
\section{The Makar-Limanov invariant}\label{S2}
In this section, we provide the basic background on the Makar-Limanov invariant and prove Proposition \ref{that} and Corollary \ref{corn}, which give further underpinning to the idea that the centre of an algebra plays a large role in whether the cancellation property holds for that algebra.  The Makar-Limanov invariant was introduced by Makar-Limanov \cite{Mak}, who called the invariant ${\rm AK}$, although it is now standard to use the terminology Makar-Limanov invariant and the notation $\ML$.

We quickly recall the basic concepts involved in the definition of this invariant. These concepts can be found in \cite{Mak, BZ1, LWZ}.
\begin{definition}\label{Hasse}
Let $k$ be a field and let $A$ be a $k$-algebra. 
\begin{enumerate}
\item We let $\Der(A)$ denote the collection of $k$-linear derivations of $A$.
 \item We let $\LND(A)=\{ \delta \in \Der(A) \mid \delta \text{ is locally nilpotent} \}$.
 \item A \emph{Hasse-Schmidt derivation} of $A$ is a sequence of $k$-linear maps $\partial := \{\partial_i\}_{i\geq 0}$ such that: \[\partial_0 = {\rm id}_A,\,\,\textrm{and}\,\, \partial_n(ab) =\sum^n_{i=0}\partial_i(a)\partial_{n-i}(b)\]
for $a, b \in A$ and $n \geq 0$. 
 \item A Hasse-Schmidt derivation $\partial=(\partial_n)$ is called ${\it locally~nilpotent}$ if
for each $a\in A$ there exists an integer $N=N(a)\geq 0$ such that $\partial_n(a) = 0$ for all $n\geq N$ and the $k$-algebra homomorphism $A[t]\to A[t]$ given by $t\mapsto t$ and $a\mapsto \sum_{i\ge 0} \partial_i(a)t^i$ is a $k$-algebra isomorphism.  If only the first condition holds then the map $A[t]\to A[t]$ is still an injective endomorphism but need not be onto; we will call Hasse-Schmidt derivations for which only the first condition holds (i.e., there exists an integer $N=N(a)\geq 0$ such that $\partial_n(a) = 0$ for all $n\geq N$) a \emph{weakly locally nilpotent Hasse-Schmidt derivation}. 
\item 
A Hasse-Schmidt derivation $\partial=(\partial_n)$ is called \emph{iterative}
if $\partial_{i} \circ \partial_{j}=\binom{i+j}{i} \partial_{i+j}$ for all $i, j \geq 0.$ The collection of Hasse-Schmidt derivations of an algebra $A$ is denoted $\Der^H(A)$ and the collection of iterative Hasse-Schmidt derivations is denoted $\Der^{I}(A)$.
The collection of locally nilpotent Hasse-Schmidt derivations (resp. iterative Hasse-Schmidt derivations, resp. weakly locally nilpotent Hasse-Schmidt derivations) of $A$ is denoted $\LND^H(A)$ (resp. $\LND^{I}(A)$, resp. $\LND^{H'}(A)$).
 \item Given $\partial:=(\partial_n) \in \Der^{H}(A),$ the kernel of $\partial$ is defined to be
\[\ker (\partial) =\bigcap_{i\geq 1}\Ker(\partial_i).\]

\item The \emph{Makar-Limanov$^{*}$ invariant} of $A$ is defined to be 
\[\ML^{*}(A)=\bigcap_{\delta \in \LND^{*}(A)} \ker(\delta).\]
\item The \emph{Makar-Limanov$^{*}$ centre} of $A$ is defined to be 
\[\ML_Z^{*}(A)=\ML^*(A)\cap Z(A).\]
 \item We say that $A$ is $\LND^{*}$-{\it rigid} (respectively {\it strongly} $\LND^{*}$-{\it rigid}) if $\ML^{*}(A)=A$ (respectively $\ML^{*}(A[t_1,\ldots, t_d])=A$, for $d\geq 1$).
\item We say that $A$ is $\LND_Z^{*}$-{\it rigid} (respectively {\it strongly} $\LND_Z^{*}$-{\it rigid}) if $\ML_Z^{*}(A)$ is equal to $Z(A)$ (respectively $\ML_Z^{*}(A[t_1,\ldots, t_d])=Z(A)$, for $d\geq 1$).
\end{enumerate}
In items (g)--(j), $*$ is either blank, $I$, $H$, or $H'$.
\end{definition}

\begin{remark} Let $k$ be a field and let $A$ be a $k$-algebra.  We recall some basic facts about derivations and Hasse-Schmidt derivations. \label{rem:auto}
\begin{enumerate}
\item If $\partial:=(\partial_n)$ is a locally nilpotent Hasse-Schmidt derivation of $A$ then by definition
the map $G_{\partial,t}: A[t]\to A[t]$ defined by
\begin{equation}\label{Gdefinition}
 a \mapsto \sum_{i=0}^{\infty} \partial_i(a)t^i,\ \textrm{for~all}~a\in A, t\mapsto t
\end{equation}
extends to a $k$-algebra automorphism of $A[t]$ and when $\partial$ is a weakly locally nilpotent Hasse-Schmidt derivation then this map is an injective endomorphism.
\item Conversely, if one has a $k$-algebra automorphism (resp. endomorphism) $G:A[t]\to A[t]$ such that $G(t)=t$ and $G(a)-a\in tA[t]$ for $a\in A$, then for $a\in A$ we have
$$G(a) = \sum_{i=0}^{\infty} \partial_i(a) t^i,$$ and $(\partial_n)$ is a locally nilpotent Hasse-Schmidt derivation (resp. weakly locally nilpotent Hasse-Schmidt derivation) of $A$ (see \cite[Lemma 2.2 (3)]{BZ1}).
 \item If the characteristic of $k$ is zero and $\delta:A\to A$ is a $k$-linear derivation, then the only iterative Hasse-Schmidt derivation $(\partial_n)$ of $A$ with $\partial_{1}=\delta$ is given by
 \begin{equation}\label{higherderivation}
     \partial_{n}=\frac{\delta^{n}}{n !}
 \end{equation}
for $n\ge 0$. This iterative Hasse-Schmidt derivation is called the \emph{canonical Hasse-Schmidt derivation} associated to $\delta.$ 
If, moreover, $\delta$ is locally nilpotent, then by \cite[Lemma 2.2(2)]{BZ1}, the map $G_{\partial, t}$ defined in item (\ref{Gdefinition}) is an automorphism and $(\partial_n)$ is a locally nilpotent iterative Hasse-Schmidt derivation, and conversely if $(\partial_n)$ is locally nilpotent then so is $\delta$. Thus locally nilpotent iterative Hasse-Schmidt derivations correspond naturally to locally nilpotent derivations in the characteristic zero case and so $\ML^{I}(A)=\ML(A)$ for algebras with characteristic zero base field.

 \item \cite[\S 1.1]{M} If the characteristic of $k$ is a positive integer $p$, then for an iterative derivation $\partial=\{\partial_{i}\}_{i\geq 0}$, $\partial_i$ can be explicitly described as \[\partial_{i}=\frac{\left(\partial_{1}\right)^{i_{0}}\left(\partial_{p}\right)^{i_{1}} \ldots\left(\partial_{p^r}\right)^{i_{r}}}{\left(i_{0}\right) !\left(i_{1}\right)! \ldots\left(i_r\right)!},\] 
 where $i=i_{0}+i_{1} p+\cdots+i_{r} p^{r}$ is the base-$p$ expansion of $i$.
In this case, an iterative Hasse-Schmidt
derivation $\partial$ is completely determined by $\partial_{1}, \partial_{p}, \partial_{p^2}, \ldots$.
\item Let $T$ be the polynomial ring $A\left[t_{1}, \ldots, t_{d}\right]$ over a $k$-algebra $A$. We
fix an integer $1 \leq i \leq d$. For each $n \geq 0$, we can define a divided power $A$-linear
differential operator $\Delta_i^n$ as follows:
\begin{equation}\label{examplech}
\Delta_{i}^{n}: t_{1}^{m_{1}} \cdots t_{d}^{m_{d}} 
\longrightarrow\left\{
\begin{array}{ll}
\binom{m_i}{n}t_{1}^{m_1} \cdots t_{i}^{m_i-n} \cdots t_{d}^{m_d} 
& {\textrm{ if } m_{i} \geq n} \\ 
{0} & {\textrm{ otherwise, }}
\end{array}\right.
\end{equation}
where $\binom{m_{i}}{n}$ is defined in $\mathbb{Z}$ or in $\mathbb{Z} /(p)$. Then $\left\{\Delta_{i}^{n}\right\}_{n=0}^{\infty}$ is a locally nilpotent iterative
Hasse-Schmidt derivation of $T$.  We can also extend an element $(\partial_n)\in \LND^{H'}(A)$ to an element  of $\LND^{H'}(T)$ by declaring that $t_1,\ldots ,t_d$ are in the kernel of $(\partial_n)$; moreover, the extension is iterative if the original Hasse-Schmidt derivation is iterative, and it is in $\LND^H(T)$ if the original weakly locally nilpotent Hasse-Schmidt derivation is in $\LND^H(A)$. Combining this observation along with data from the maps $\Delta_i^n$,  we see $$\ML^*(A[t_1,\ldots ,t_d])\subseteq \ML^*(A),$$ where $*$ is either $I$, $H$, or $H'$.
\end{enumerate}
\end{remark}
We make use of the following definitions from \cite{BZ1, LWZ}.  
\begin{definition}
Let $A$ be an algebra.
\begin{enumerate}
\item We call $A$ \emph{cancellative} (respectively \emph{strongly cancellative}) if an isomorphism $A[t]\to B[t]$ with $B$ an algebra (respectively, for each $d\ge 1$, an isomorphism $A[t_1,\ldots ,t_d] \to B[t_1,\ldots ,t_d]$ with $B$ an algebra) implies that $A\cong B$.
\item We call $A$ \emph{retractable} (respectively \emph{strongly retractable}) if an isomorphism $\phi: A[t]\to B[t]$ for some algebra $B$ (respectively an isomorphism $\phi: A[t_1,\ldots ,t_d] \to B[t_1,\ldots ,t_d]$ for some integer $d\ge 1$ and some algebra $B$) implies that $\phi(A)=B$.
\item We call $A$ $Z$-\emph{retractable} (respectively \emph{strongly} $Z$-\emph{retractable}) if, given an algebra $B$ and an algebra isomorphism $\phi: A[t] \cong B[t]$  (respectively an isomorphism $\phi: A[t_1,\ldots ,t_d] \to B[t_1,\ldots ,t_d]$ for some integer $d\ge 1$ and some algebra $B$), we necessarily have $\phi(Z(A)) = Z(B)$.
\item We call $A$ \emph{detectable} (respectively \emph{strongly detectable}) if, given an algebra $B$ and an isomorphism $\phi : A[t]\to B[s]$ (respectively given an algebra $B$, $d\ge 1$, and an isomorphism $\phi : A[t_1,\ldots ,t_d]\to B[s_1,\ldots ,s_d]$), we necessarily have $s\in B\{\phi(t)\}$ (resp. $s_i\in B\{\phi(t_1),\ldots ,\phi(t_d)\} ~~~{\rm for}~i=1,\ldots ,d$).

\end{enumerate}
\end{definition}

We begin by proving a lemma, which is the counterpart of \cite[Lemma 3.2]{BZ1}.
\begin{lemma}
Let $Y := \bigoplus_{i=0}^{\infty}Y_i$ be an $\mathbb{N}$-graded $k$-algebra and suppose that $Y_0 y Y_0$ contains a regular element whenever $y$ is a nonzero homogeneous element of $Y$. If $Z$ is a subalgebra of $Y$
containing $Y_0$ such that ${\rm GKdim}(Z) = {\rm GKdim}(Y_0)<\infty$, then $Z=Y_0$.
\label{lem:BZ}\end{lemma}
\begin{proof}
Suppose that $Z$ strictly contains $Y_0$ as a subalgebra.  Since $Y$ is a graded algebra, $Z$ is an $\mathbb{N}$-filtered algebra with $F_0Z = X$. By \cite[Lemma 6.5]{KL}, ${\rm GKdim}(Z)\ge {\rm GKdim}({\rm gr}(Z))$, where $
{\rm gr}(Z)$ is the associated graded ring of $Z$ with respect to the filtration induced by the $\mathbb{N}$-grading on $Y$. Then ${\rm gr}(Z)$ is an $\mathbb{N}$-graded subalgebra of $Y$ that strictly contains $Y_0$ as the degree 
zero part, and so ${\rm gr}(Z)$ contains some nonzero homogeneous element $y\in Y_d$ for some $d\ge 1$.  Then it contains the $Y_0$-$Y_0$-bimodule $Y_0 y Y_0\subseteq Y_d$.  In particular, there is some regular 
homogeneous element $a\in Z$ of positive degree and so by considering the grading we have $$Y_0+Y_0a+\cdots $$ is direct and is contained in ${\rm gr}(Z)$.  From this one can easily show that $${\rm GKdim}({\rm gr}(Z))\ge {\rm 
GKdim}(({\rm gr}(Z))_0) +1 \ge {\rm GKdim}(Y_0)+1.$$ Combining these inequalities gives $${\rm GKdim}(Z)\ge {\rm GKdim}(Y_0)+1,$$ a contradiction.  Thus $Z=Y_0$.  
\end{proof}
We will use Lemma \ref{lem:BZ} in the case when $A$ is a prime left Goldie algebra and $Y=A[t_1,\ldots ,t_d]$, where we declare that elements of $A$ have degree $0$, and $t_1,\ldots ,t_d$ are homogeneous of degree $1$.  Observe that if $p(t_1,\ldots,t_d)$ is a nonzero homogeneous polynomial of degree $m$ in $Y$, then we can put a degree lexicographic order on the monomials in $t_1,\ldots ,t_d$ by declaring that $t_1>t_2>\cdots > t_d$. Then we let $t_1^{i_1}\cdots t_d^{i_d}$ denote the degree lexicographically largest monomial that occurs in $p(t_1,\ldots ,t_d)$ with nonzero coefficient and we let $a\in A$ denote this coefficient.  Then since $A=Y_0$ is prime Goldie, $Y_0 a Y_0$ contains a regular element, and so $Y_0 p(t_1,\ldots ,t_d) Y_0$ contains a nonzero homogeneous polynomial $q=q(t_1,\ldots ,t_d)$ such that the degree lexicographically largest monomial that occurs in $q(t_1,\ldots, t_d)$ with nonzero coefficient has the property that this coefficient is regular; moreover, this monomial is again $t_1^{i_1}\cdots t_d^{i_d}$, and we let $c\in A$ denote this coefficient.  We now claim that $q$ must be regular.  To see this, let $h$ be a nonzero polynomial in $Y$.  Then let $t_1^{j_1}\cdots t_d^{j_d}$ denote the degree lexicographically largest monomial that occurs in $h$ with nonzero coefficient, and let $b\in A$ denote this coefficient.  Then by construction the coefficient of $t_1^{i_1+j_1}\cdots t_d^{i_d+j_d}$ in $qh$ is $cb$ and since $b$ is nonzero and $c$ is regular, $qh\neq 0$; similarly, $hq\neq 0$ and so $q$ is regular.  In particular, $Y$ satisfies the hypotheses of Lemma \ref{lem:BZ}, in this case, which we will now apply in the following proposition.
\begin{proposition}\label{M3}
Let $A$ be a finitely generated prime left Goldie $k$-algebra of finite Gelfand-Kirillov dimension. 
Let $*$ be either blank, $H$, $H'$ or $I$.  When $*$ is blank we further assume $k$ has characteristic zero.
\begin{enumerate}
 \item If $\ML^{*}(A[t])=A$, then $A$ is retractable and so is cancellative.
 \item If $\ML^{*}(A[t_1,\ldots,t_n])=A$, then $A$ is strongly-retractable and so is strongly cancellative. 
 \item {\rm (}\cite[Lemma 2.6]{LWZ}{\rm )} Suppose $Z(A)$ is affine and $\ML_Z(A[t]) = Z(A)$ or $\ML^H_Z (A[t]) = Z(A)$ or $\ML^{H'}_Z(A[t])=Z(A)$. Then $A$ is $Z$-retractable.
 \item {\rm (}\cite[Lemma 2.6]{LWZ}{\rm )} Suppose $Z(A)$ is affine and $A$ is strongly $\LND_Z^*$-rigid where $*$ is either blank, $H$, or $H'$. Then $A$ is strongly $Z$-retractable.
\end{enumerate}
\end{proposition}
\begin{proof}
The proof is identical to the proof given in \cite[Theorem 3.3]{BZ1}, with the one exception being that we invoke Lemma \ref{lem:BZ} with $Y=A[t_1,\ldots ,t_d]$ (with elements of $A$ having degree $0$ and $t_1,\ldots ,t_d$ having degree $1$) as a replacement for Lemma 3.2 used in \cite{BZ1}.  We point out that \cite{LWZ} do not use $H'$ in their paper, but the argument in the $H'$ case goes through in the same manner as it does for $H$.
\end{proof}

In Proposition \ref{that} below, we give a result that is related to a conjecture of Makar-Limanov \cite[p. 55]{Ma1}, which has interesting implications in terms of cancellation.  To prove this result, we need to invoke a result of \cite{LWZ} that requires that the algebras involved by strongly Hopfian.  An algebra $A$ is \emph{strongly Hopfian} if whenever $d\ge 1$ and $\phi$ is a surjective endomorphism of $A[t_1,\ldots ,t_d]$, $\phi$ is necessarily injective. Prime Goldie algebras of finite Gelfand-Kirillov dimension are strongly Hopfian (see Krause and Lenagan \cite[Proposition 3.15]{KL}), which is used in the proof.
We first need a basic result about vanishing of polynomials in noncommutative rings.
\begin{remark}\label{rem: zero} Let $A$ be a prime ring and let $p(x)\in A[x]$ be a nonzero polynomial of degree $d$. If there are $d+1$ distinct central elements $z\in A$ such that $p(z)=0$ then $p(x)$ is the zero polynomial.
\end{remark}
\begin{proof} Write $p(x)=a_0+a_1 x+\cdots +a_d x^d$.  Let $Z$ denote the centre of $A$, which is an integral domain since $A$ is prime.  Suppose that there exist distinct $z_1,\ldots ,z_{d+1}\in Z$ such that $p(z_i)=0$ for $i=1,\ldots, d+1$.  Let $M$ be the $(d+1)\times (d+1)$ matrix whose $(i,j)$-entry is $z_j^{i-1}$.  Then considering $A^{d+1}$ as a right $M_{d+1}(Z)$-module, we see $[a_0,a_1,\ldots ,a_d]M = 0$.  Then right-multiplying by the classical adjoint of $M$ we obtain $a_i \det(M)=0$ for $i=0,\ldots ,d$.  Then $M$ is a Vandermonde matrix and since $Z$ is an integral domain and the $z_i$ are pairwise distinct, $\det(M)$ is a nonzero central element of $A$, and hence is regular since $A$ is prime.  It follows that $a_0=\cdots =a_d=0$ and $p(x)$ is the zero polynomial.
\end{proof}

\begin{proposition}\label{chrigid}
Let A be a prime finitely generated $k$-algebra with infinite centre. Then $\ML^{H'}(A)=\ML^{H'}(A[x_1, x_2, \cdots, x_d]).$ In particular, if, in addition, $A$ is left Goldie, has finite Gelfand-Kirillov dimension, $Z(A)$ is affine, and either $\ML^{H'}_Z(A)=Z(A)$ or $\ML^{H'}(A)=A$, then $A$ is strongly cancellative.
\label{that}
\end{proposition}
\begin{proof}
Remark \ref{rem:auto} gives that $\ML^{H'}(A[x_1, x_2, \ldots, x_d])\subseteq \ML^{H'}(A)\subseteq A$ for all $d\geq 1$. It thus suffices to show that $\ML^{H'}(A)\subseteq \ML^{H'}(A[x_1, x_2, \ldots, x_d]).$

We show that $\ML^{H'}(A)\subseteq \ML^{H'}(A[x])$. Once we have proved this, it will immediately follow by induction that $\ML^{H'}(A)\subseteq \ML^{H'}(A[x_1,\ldots ,x_d])$ and we will obtain the result.
Let $\partial:=\{\partial_n\}_{n\geq 0}$ be an element of $\LND^{H'}(A[x])$.  As in Equation (\ref{Gdefinition}), we have an induced $k$-algebra homomorphism $\phi: A[t]\longrightarrow A[x][t]$, given by \begin{equation*}
\begin{aligned}
 \phi(a)=\sum_{n\ge 0}\partial_n(a)t^n~~ \textrm{for}\,\, a\in A, ~~\phi(t)=t.
\end{aligned}
\end{equation*}
In particular, if $a\in \ML^{H'}(A)$, then $\phi(a)=a+t p(x,t)$, for some polynomial $p(x,t)\in A[x,t]$.  We now fix $z\in Z(A)$ and consider the map $e_z : A[x,t]\to A[t]$, defined by $e_z(g(x,t))=g(z,t)$.
Then the composition $\phi_z:=e_z\circ \phi$ gives a homomorphism from $A[t]$ to $A[t]$ and by construction $\phi_z(a)\equiv a~(\bmod~(t))$, and $\phi(t)=t$ and so this homomorphism is injective.  Thus there are maps $\mu_j :A \to A$ with $\mu_0={\rm id}_A$ such that 
$\phi_z(a) = \sum_{j\ge 0} \mu_j(a) t^j$ for $a\in A$.  In particular for $a\in A$, $\mu_n(a)=0$ for $n$ sufficiently large, and so $(\mu_n)$ is a weakly locally nilpotent Hasse-Schmidt derivation of $A$.  Thus for $a\in \ML^{H'}(A)$ we have $\mu_i(a)=0$ for every $i\ge 1$; that is, for $i\ge 1$, $\partial_i(a)|_{x=z}=0$ for every $z\in Z(A)$.  Since $Z(A)$ is infinite and $\partial_n(a)$ is a polynomial in $A[x]$, Remark \ref{rem: zero} gives that $\partial_n(a)=0$ for $n\ge 1$ and hence $a\in \ML^{H'}(A[x])$. Thus $\ML^{H'}(A)\subseteq \ML^{H'}(A[x])$ as required.

Now suppose that $Z(A)$ is infinite and affine and that $A$ is prime left Goldie and has finite Gelfand-Kirillov dimension. It follows that if $\ML_Z^{H'}(A)=Z(A)$ then from the above $ML_Z^{H'}(A[x_1,\ldots ,x_d])=Z(A)$ and so $A$ is strongly $\LND_Z^{H'}$-rigid and hence strongly $Z$-retractable by Proposition \ref{M3}. Thus by \cite[Lemma 3.2]{LWZ}, $A$ is strongly detectable, and so is strongly cancellative \cite[Lemma 3.6]{LWZ}, since $A$ is strongly Hopfian (cf. Krause and Lenagan \cite[Proposition 3.15]{KL}). On the other hand if $\ML^{H'}(A)=A$ then $A$ is strongly $\LND^{H'}$-rigid and so by Proposition \ref{M3}, $A$ is strongly cancellative.
\end{proof}

In analogy with terminology from algebraic geometry, given an algebraically closed field $k$ and a finitely generated extension $F$ of $k$, we will say that $F$ is \emph{uniruled} over $k$ if there is a finitely generated field extension $E$ of $k$ with ${\rm trdeg}_k(E)={\rm trdeg}_k(F)-1$ and an injective $k$-algebra homomorphism $F\to E(t)$. The idea here is that $F$ is the function field of a normal projective scheme $X$ of finite type over $k$.  Then the condition $F\subseteq E(t)$ says that there is a dominant rational map $Y\times \mathbb{P}^1\dashrightarrow X$ for some variety $Y$ with ${\rm dim}(Y)={\rm dim}(X)-1$.  Similarly, if $F$ is the function field of a normal projective scheme $X$ of finite type over $k$, we define the \emph{Kodaira dimension} of $F$ to be the Kodaira dimension of $X$.  Since Kodaira dimension is a birational invariant, this is well-defined.  We refer the reader to Hartshorne \cite{Har} for background in algebraic geometry and on Kodaira dimension.  If $k$ has characteristic zero, a uniruled variety has Kodaira dimension $-\infty$ and the converse holds in dimensions one, two, and three; the main conjectures of the minimal model program imply that the converse should hold in higher dimensions, too. 

Over uncountable fields, affine uniruled varieties have a pleasant characterization in terms of being covered by affine lines (see \cite{Jel, Sta}).
\begin{proposition}
Let $k$ be an uncountable algebraically closed field and let $X$ be an irreducible affine variety over $k$ of dimension at least one. Then
following conditions are equivalent:
\begin{enumerate}
\item for every $x\in X$ there is a polynomial affine curve $Y_x$ in $X$ that passes through $x$;
\item there is a Zariski-dense open subset $U$ of $X$, such that for every $x\in U$ there is a polynomial affine curve $Y_x$ in $X$ that passes through $x$;
\item $X$ is uniruled; that is, there exists an affine variety $Y$ with ${\rm dim}(Y)={\rm dim}(X)-1$ and a dominant morphism
$Y\times \mathbb{A}^1\to X$.
\end{enumerate}
\label{prop:uniruled}
\end{proposition}
\begin{corollary} Let $k$ be an uncountable algebraically closed field, let $A$ be a finitely generated prime left Goldie $k$-algebra of finite Gelfand-Kirillov dimension, and suppose that $Z(A)$ is affine.  If $A$ does not possess the strong cancellation property then ${\rm Frac}(Z(A))$ is uniruled.   In particular, if $k$ has characteristic zero and ${\rm Frac}(A)$ has nonnegative Kodaira dimension then $A$ is strongly cancellative. \label{corn}
\end{corollary}
\begin{proof}
We claim that if $Z(A)$ has a non-trivial weakly locally nilpotent Hasse-Schmidt derivation then the field of fractions of $Z(A)$ is necessarily uniruled.  To see this, suppose that $\partial:=(\partial_n)$ is a non-trivial weakly locally nilpotent Hasse-Schmidt derivation of $Z(A)$. Then by Remark \ref{rem:auto} we have an injective 0 $G_{\partial} : Z(A)[t]\to Z(A)[t]$ that sends $t$ to $t$, and by assumption $G_{\partial}$ is not the identity on $Z(A)$.  Let $X={\rm Spec}(Z(A))$, which is an affine scheme of finite type over $k$.
Then the induced $k$-algebra homomorphism $$Z(A)\to Z(A)\otimes _k k[t]=Z(A)[t]$$ from $G_{\partial}$ yields a morphism $\Phi: \mathbb{A}^1\times X\to X$ with $\Phi(0,x)=x$ for $x\in X$, and since $(\partial_n)$ is non-trivial, there is some $x\in X$ such that $\Phi(\mathbb{A}^1\times \{x\})$ is not a point.  We claim there is a Zariski dense open set $U\subseteq X$ such that for $x\in U$ we have 
$\Phi(\mathbb{A}^1\times \{x\}) = Y_x\subseteq X$ with $Y_x$ birationally isomorphic to $\mathbb{P}^1$.  To see this, notice that for each $x\in X$, $\Phi$ gives a map from $\mathbb{A}^1\to Y_x$.  Since $\mathbb{A}^1$ is one-dimensional and irreducible, $Y_x$ is either a point or an irreducible curve.  Moreover, if $Y_x$ is a curve, then we have a dominant rational map $\mathbb{P}^1\dashrightarrow Y_x$ and so $Y_x$ is birationally isomorphic to $\mathbb{P}^1$ by L\"uroth's theorem.  So now let $V$ denote the set of $x\in X$ such that $Y_x$ is a point.  Now there is at least one point $x\in X$ such that $Y_x$ is infinite, so pick $p,q\in \mathbb{A}^1$ and $x_0\in X$ such that $\Phi(p,x_0)\neq \Phi(q,x_0)$.  Then 
$\Psi: X\to X\times X$ given by $x\mapsto (\Phi(p,x), \Phi(q,x))$ is a morphism and since the diagonal $\Delta$ is closed in $X\times X$, $Y:=\Psi^{-1}(\Delta)$ is a closed subvariety of $X$ and by assumption $x_0\not\in Y$ and so $Y$ is proper.  Thus $U:=X\setminus Y$ has the property that $\Phi(p,x)\neq \Phi(q,x)$ for $x\in U$ and so $Y_x$ is necessarily a rational curve for $x\in U$.  Thus $X$ has an open dense subset such that each $k$-point in $U$ is covered by rational curves and so $X$ is uniruled by Proposition \ref{prop:uniruled}. In particular, there is a dominant rational map from a variety of the form $Y\times \mathbb{P}^1$ to $X$, where ${\rm dim}(Y)={\rm dim}(X)-1$.  Thus ${\rm Frac}(A)= k(X)\hookrightarrow k(Y\times \mathbb{P}^1) \cong k(Y)(t)$ and so $F$ is uniruled.  Thus if $F$ is not uniruled then $\ML^{H'}(Z(A))=Z(A)$ and since an element of $\LND^{H'}(A)$ restricts to an element of $\LND^{H'}(Z(A))$, $\ML^{H'}_Z(A)=Z(A)$ and so $A$ is strongly cancellative by Proposition \ref{chrigid}. 
\end{proof}

The above result shows under certain conditions that if the centre of an algebra is sufficiently ``rigid'' then the algebra is strongly cancellative.  We conjecture that over ``nice'' base fields the centre completely determines cancellation.  We make this precise with the following conjecture. 
\begin{conjecture} Let $k$ be an uncountable algebraically closed field of characteristic zero and let $A$ be an affine noetherian domain over $k$.  Suppose that $Z(A)$ is affine and cancellative (respectively strongly cancellative).  Then $A$ is cancellative  (resp. strongly cancellative).\label{conic}
\end{conjecture}

\section{The noncommutative slice theorem and proof of Theorem \ref{thm: main} ($a$)}\label{S3}
The slice theorem (see the book of Freudenburg \cite[Theorem 1.26]{Fr}) is a powerful tool when working on the Zariski cancellation problem. Roughly speaking, it says that if one has a locally nilpotent derivation $\delta$ of a commutative algebra $R$ of characteristic zero and a regular element $r\in R$ such that $\delta(r)=1$, then $R\cong S[t]$ where $S$ is the kernel of $\delta$.  We shall prove a noncommutative analogue of this result. We make use of the fact that for a $k$-algebra $A$ with a locally nilpotent derivation $\delta$, the map $\delta$ restricts to a locally nilpotent derivation of the centre of $A$.

This following lemma is an extension of the slice theorem for a (not necessarily commutative) prime affine $k$-algebra.  
\begin{lemma} (Noncommutative slice theorem)\label{gsliceth}
Let $k$ be a field and let $A$ be a $k$-algebra. Then the following statements hold.
\begin{enumerate}
 \item Suppose that the characteristic of $k$ is zero and $\delta\in \LND(A)$. If there exists $x\in Z(A)$ such that $\delta(x)=1$, and if $A_0$ is the kernel of $\delta$, then the sum $\sum_{i\ge 0} A_0x^i$ is direct and $A=A_0[x]$.
 \item Suppose that $\partial:=\{\partial_n\}_{n\geq 0}\in \LND^{I}(A)$. If there exists $x\in Z(A)$ such that $\partial_1(x)=1$ and $\partial_i(x)=0$ for $i\geq 2$, and if $A_0$ is the kernel of $\partial$, then the sum $\sum_{i\ge 0} A_0x^i$ is direct and $A=A_0[x]$.\end{enumerate}
\end{lemma}
Before giving the proof of this result, we first make a basic remark.
\begin{remark} Let $k$ be a field, let $A$ be a prime $k$-algebra, let $\partial:=(\partial_n)\in \LND^I(A)$, and let $B={\rm ker}(\partial)$.   Then the following hold:
\begin{enumerate}
\item if there is $x\in A$ and $m\ge 1$ are such that $\partial_{m+i}(x)=0$ for $i\ge 1$ and $\partial_m(x)$ is a regular element of $A$, then the sum $$B+Bx+Bx^2+\cdots $$ is direct;
\item if $A$ is a field and $A$ is algebraic over $B$ then $A=B$;
\item if ${\rm GKdim}(A)<\infty$ and $A$ is an affine domain and $B\neq A$ then ${\rm GKdim}(B)\le {\rm GKdim}(A)-1$;
\item if ${\rm GKdim}(A)=1$ and $A$ is an affine domain and $B\neq A$ then $B$ is finite-dimensional.
\end{enumerate}
\label{rem:direct}

\end{remark}
\begin{proof} Suppose there exist $x\in A$ and $m\ge 1$ such that $\partial_{m+i}(x)=0$ for $i\ge 1$ and $\partial_m(x)$ is a regular element of $A$.
A straightforward induction shows that $\partial_{ms}(x^i)=0$ for $i<s$ and $\partial_{ms}(x^s) = \partial_m(x)^s$.  Suppose that there is a non-trivial relation $b_0+b_1x+\cdots + b_s x^s =0$ with $b_0,\ldots ,b_s\in B$ and $b_s$ nonzero.  Then applying $\partial_{ms}$ to this dependence gives $b_s\partial_m(x)^s=0$, which is impossible as $b_s\neq 0$ and $\partial_m(x)$ is regular.  This establishes (a).
Next, to prove (b), observe that if $A$ is a field, then $B$ is a subfield of $A$.  We have just shown that for $x\in A\setminus B$, the sum $B+Bx+\cdots $ is direct, and so if $A$ is algebraic over $B$ then we must have $A=B$.

We next prove (c). Suppose that $A$ is a domain of finite Gelfand-Kirillov dimension and that $B\neq A$ and that
$${\rm GKdim}(B)>{\rm GKdim}(A)-1.$$  Then there exists $\alpha>{\rm GKdim}(A)-1$ and a finite-dimensional $k$-vector subspace $W$ of $B$ that contains $1$ and such that
${\rm dim}(W^n) \ge n^{\alpha}$ for $n$ sufficiently large.  Pick $x\in A\setminus B$.  Then by (a), $B+Bx+Bx^2+\cdots $ is direct.  Now let $V=W+kx$.
Then $V^{2n} \supseteq W^n + W^n x +\cdots + W^n x^n$ and so ${\rm dim}(V^{2n})\ge (n+1) n^{\alpha} \ge n^{\alpha+1}$.  Thus ${\rm GKdim}(A)\ge \alpha+1$, a contradiction.  Thus we obtain (c).

We now prove (d).  Suppose that $A$ is an affine domain of Gelfand-Kirillov dimension one and that $B\neq A$. We claim that ${\rm dim}_k(B)<\infty$. Pick $z\in A\setminus B$. By part (a), the sum $B+Bz + Bz^2 +\cdots $ is direct.  Now suppose towards a contradiction that ${\rm dim}_k(B)$ is infinite and let $V$ be a finite-dimensional subspace of $A$ that contains $1$ and $z$ and which generates $A$ as a $k$-algebra.  Then since
$\bigcup_{i\ge 0} V^i \supseteq B$, we have $W_n:= V^n\cap B$ has the property that ${\rm dim}(W_n)\to\infty$ as $n\to\infty$.  Since $A$ has Gelfand-Kirillov dimension one, by a result of Bergman (see the proof of \cite[Theorem 2.5]{KL}) there is some positive constant $C$ such that 
${\rm dim}(V^n)\le Cn$ for $n$ sufficiently large.  On the other hand, for each $d\ge 1$ we have
$${\rm dim}(V^{n+d}) \ge {\rm dim}(W_d V^n) \ge {\rm dim}(W_d + W_d z + \cdots + W_d z^n) = {\rm dim}(W_d)(n+1).$$
Thus ${\rm dim}(W_d) \le C(n+d)/(n+1)$ for all $n$ sufficiently large and so ${\rm dim}(W_d) \le C$ for every $d\ge 1$, a contradiction.  Thus $B$ is finite-dimensional.
\end{proof}

\begin{proof}[Proof of Lemma \ref{gsliceth}]
It suffices to prove part (b) by Remark \ref{rem:auto}. We let 
$$A_0 = \{a \in A \mid \partial_n(a) = 0 ~ \textrm{for} ~n\geq 1\}.$$ 
We claim that $A=A_0[x]$. By Remark \ref{rem:direct}, $\sum A_0 x^i$ is direct. 
Thus $A_0$ and $x$ generate a polynomial ring and $A\supseteq A_0[x]$.  We next claim that $A\subseteq A_0[x]$.  To see this, suppose that there exists some $a\in A\setminus A_0[x]$.  Then there is some largest $m\ge 1$ such that $\partial_m(a)\neq 0$.  Among all $a\in A\setminus A_0[x]$, we choose one with this $m$ minimal.  Since $\partial_i(\partial_m(a))={i+m\choose m}\partial_{i+m}(a)=0$ for $i\ge 1$, $\partial_m(a)$ is in the kernel of $\partial$ and hence in $A_0$.  Let $c=\partial_m(a)\in A_0$ and consider $a'=a-cx^m$.  Observe that $\partial_j(a')=0$ for $j> m$ and $\partial_m(a')=0$ by construction.  Thus by minimality of $m$, $a'\in A_0[x]$ and hence so is $a$, a contradiction.  The result follows.
\end{proof}
\begin{proposition}\label{ggslice}
Let $k$ be a field of characteristic zero and let $A$ be a prime finitely generated $k$-algebra of finite Gelfand-Kirillov dimension, and suppose that $Z(A)$ is an affine domain of Gelfand-Kirillov dimension at most $1$. Then one of the following alternatives must hold:
\begin{enumerate}
\item ${\rm ML}_Z(A)=Z(A)$; or
\item there is a prime $k$-subalgebra $A_0$ of $A$ such that $A\cong A_0[t]$.
\end{enumerate}
\end{proposition}
\begin{proof} If ${\rm ML}_Z(A)\neq Z(A)$, then there is some $\delta\in \LND(A)$ and some $z\in Z(A)$ such that $\delta(z)\neq 0$.  
We now pick the largest $j$ such that $\delta^j(z)\neq 0$ and we replace $z$ by $\delta^{j-1}(z)$. By construction, $\delta^i(z)=0$ for $i\ge 2$ and $c:=\delta(z)\neq 0$.  Then $c\in A_0\cap Z(A)$.  Now $A_0\cap Z(A)$ is a subalgebra of $Z(A)$ and since $Z(A)$ has Krull dimension one and $A_0\subsetneq Z(A)$, $A_0$ is finite-dimensional by Remark \ref{rem:direct} and thus is a field.  Thus $c$ is a unit and so if we replace $z$ by $c^{-1}z$ then we have $\delta(z)=1$ and we may invoke Lemma \ref{gsliceth} to get that $A\cong A_0[t]$.  Since $A$ is prime, $A_0$ is necessarily prime too. \end{proof}
In general, the proof of Proposition \ref{ggslice} shows that if $A$ is a affine prime $k$-algebra of finite Gelfand-Kirillov dimension over a field $k$ of characteristic zero, then either $\ML_Z(A)=Z(A)$ or there is a prime subalgebra $A_0$ of $A$ and some $c\in Z(A)\cap A_0$ such that $A[c^{-1}]\cong A_0[c^{-1}][t]$.  In the case, that $Z(A)$ is affine of Gelfand-Kirillov dimension one we are able to deduce that $c$ is invertible in the proof, which gives us part (b) in the dichotomy occurring in Proposition \ref{ggslice}.
\begin{proof}[Proof of Theorem \ref{thm: main} (a)] We recall that affine prime algebras of Gelfand Kirillov dimension one are noetherian and hence left Goldie \cite{SW}.
If $\ML(A)=A$ then $A$ is cancellative by Proposition \ref{M3}.  If on the other hand, $\ML(A)\neq A$, then there is a nonzero locally nilpotent derivation $\delta$ of $A$.  Let $A_0$ denote the kernel of $\delta$.  Then by Remark \ref{rem:direct} $A_0$ is finite-dimensional and since it is a domain, it is a division ring.

In particular, $Z(A)\not\subseteq A_0$, since $Z(A)$ has Gelfand-Kirillov dimension one \cite{SW}.  We let $E=A_0\cap Z(A)$.  Then $E$ is a commutative integral domain that is finite-dimensional over $k$ and hence $E$ is a field.  Since $\delta$ is not identically zero on $Z(A)$ and is locally nilpotent on $A$, there exists some $z\in Z(A)$ such that $z\not\in E$ and $c:=\delta(z)\in E\setminus \{0\}$.  As $E$ is a field and is contained in the kernel of $\delta$,  $x:=c^{-1}z\in Z(A)$ satisfies $\delta(x)=1$ and so by the noncommutative slice theorem, we see $A\cong A_0[x]$.  Then by the same analysis as above if $A[t]\cong B[t]$ then we necessarily have $\ML(B)\neq B$ and so $B\cong B_0[x]$ for some finite-dimensional division ring $B_0$.  Since $A_0$ is a finite-dimensional division algebra, it follows from \cite[Theorem 4.1]{LWZ} that $A_0$ is strongly cancellative and hence $A_0\cong B_0$ and hence $A$ is cancellative.  Thus we obtain the result in this case.
\end{proof}

We next prove a result, which has rather technical hypotheses, although we believe the result is important in understanding cancellation in positive characteristic.  Given an affine domain $A$ over a field $k$, we say that $k$ is \emph{inseparably closed} in $A$ if whenever $F$ is a $k$-subalgebra of $A$ that is a field, we have that $F$ is separable over $k$.  In particular, when $k$ has characteristic zero, this always holds. Throughout this proof we make use of the so-called Lucas identity, which says that if $p$ is prime and $0\le a,b<p$ and $A,B>0$ then
$$\binom{pA+a}{pB+b}\equiv \binom{A}{B}\binom{a}{b}~(\bmod \, p).$$
\begin{proposition} Let $k$ be a field and let $A$ be an affine domain of Gelfand-Kirillov dimension one.  Suppose that $k$ is inseparably closed in $A$ and that $\ML^I(A)\neq A$. Then $A$ is strongly cancellative.\label{pure}
\end{proposition}
\begin{proof}
When the characteristic of $k$ is zero, this follows immediately from Theorem \ref{thm: main} (a). Thus we may assume that $k$ has characteristic $p>0$.
Fix a non-trivial locally nilpotent iterative Hasse-Schmidt derivation $(\partial_n)$ of $A$ and let $D$ denote the kernel of $(\partial_n)$.  Then by Remark \ref{rem:direct}, $D$ is finite-dimensional and hence a finite-dimensional division ring over $k$. Given nonzero $a\in A$, we define $\nu(a) = \sup\{ m\colon \partial_m(a)\neq 0\}$.
Then we pick $a\in A\setminus D$ with $m:=\nu(a)$ minimal among elements of $A\setminus D$.  We claim that $\nu(a)=p^r$ for some $r\ge 0$.  To see this, suppose that this is not the case.
Then $m=\nu(a)=p^r s_0 + p^{r+1} s_1$ with $r\ge 0$, $1\le s_0<p$, and $s_0+ps_1>1$. If $s_1\ge 1$, then observe that $b=\partial_{p^r s_0}(a)$ has the property that
$$\partial_{p^{r+1}s_1}(b) = \partial_{p^{r+1}s_1}\circ \partial_{p^r s_0}(a) =  \binom{p^{r+1}s_1+p^rs_0}{p^r s_0}\partial_{m}(a) = \binom{ps_1+s_0}{s_0} \partial_{m}(a)\neq 0$$
and for $i>p^{r+1} s_1$, we have $\partial_i(b) \in k\partial_{m+i-p^{r+1}s}(a)=\{0\}$ and hence $\nu(b)=p^{r+1}s_1 < m$.  If, on the other hand, $s_1=0$, we have $m=p^r s_0$ with $2\le s_0<p$.  Then
if we let $b=\partial_{p^r}(a)$, then as before we have $\partial_i(b)=0$ for $i>p^r (s_0-1)$ and $\partial_{p^r(s_0-1)}(b)\neq 0$.  It follows that $m=\nu(a)$ is necessarily of the form $p^r$ for some $r\ge 0$.
Let $\alpha=\partial_m(a)$.  Then for $i\ge 1$, $$\partial_i(\alpha)=\partial_i\circ \partial_m(a) = \binom{m+i}{i}\partial_{m+i}(a)=0$$ and hence $\alpha\in D\setminus \{0\}$.  Then by replacing $a$ by $\alpha^{-1}a$, we may assume without loss of generality that $\alpha=1$.

We next claim that $p^r=\nu(a)$ divides $\nu(b)$ for every $b\in A$.  To see this, suppose this is not the case.  Then there exists some $b\in A$ such that $\nu(b) = p^r s+ q$ with $0<q<p^r$.  Consequently, there is some $i<r$ such that $q=p^i q'$ with $\gcd(q',p)=1$ and $q'<p^{r-i}$.  We let $b'=\partial_{p^r s}(b)$, and we have
$$\partial_{q}(b') = \partial_q \partial_{p^rs}(b) = \binom{p^r s + p^i q'}{p^i q'} \partial_m(b) = \binom{p^{r-i} s + q'}{q'}\partial_m(b)\neq 0.$$
Also for $i>q$ we have $\partial_i(b') \in k\partial_{m+i-q}(b)=\{0\}$ and so $0<\nu(b')=q<m$, which contradicts minimality of $m$.  

We now prove that $A\cong D[x]$.  To see this, observe that for $\beta\in D$, $\partial_i([\beta, a]) = [\beta, \partial_i(a)] =0$ for $i>m$ and since $\partial_m(a)=1$, $\partial_m([\beta,a])=0$ for $\beta\in D$ and thus $\nu([\beta,a])<m$ for all $\beta\in D$.  By minimality of $m$, $\nu([\beta,a])=0$ for $\beta\in D$ and so $[D,a]\subseteq D$. Hence the map $\delta: D\to D$ given by $\delta(\beta)=[\beta,a]$ is a $k$-linear derivation of $D$.  We claim that $A=D+Da+\cdots $.  Since $D\subseteq A$ and $a\in A$, it suffices to show that $A$ is contained in $$\sum Da^i.$$
So suppose that this containment does not hold.  Then there is some $$b\in A\setminus (D+Da+Da^2+\cdots ).$$  Among all such $b$, pick one with $\nu(b)$ minimal.  From the above we have $\nu(b)=p^r s=ms$ for some $s\ge 1$.  Let $\gamma=\partial_{p^rs}(b)\in D$ and consider $b':=b-\gamma a^s$.  By construction, $\nu(b')<\nu(b)$ and so by minimality of $\nu(b)$, $b'\in D+Da+\cdots$, which then gives that $b$ is too, a contradiction.  It follows that $A=D+Da+\cdots$ and since $A$ is infinite-dimensional and $D$ is finite-dimensional, this sum is direct; moreover, $[a,\beta]=\delta(\beta)$ for $\beta\in D$, and so $A\cong D[x;\delta]$ with $\delta$ a $k$-linear derivation and $k$ contained in $Z(D)$ and $[D:k]<\infty$.  Now by assumption $Z(D)$ is separable over $k$ and so $\delta$ vanishes on $Z(D)$ \cite[Prop. 3, V. p. 128]{Our}.  Thus $\delta$ is a $Z(D)$-linear derivation of $D$ and by a straightforward application of the Skolem-Noether theorem it is thus inner \cite[Theorem 3.22]{BF}.  Hence by making a change of variables of the form $x'=x-c$ with suitably chosen $c\in D$, we have $A\cong D[x']$.  But now $D$ is strongly cancellative \cite[Theorem 4.1]{LWZ} and thus $A$ is strongly cancellative. 
\end{proof}
\section{Examples}\label{S4}
In this brief section, we give a family of examples that establish Theorem \ref{thm: main} (b).  

\begin{proof}[Proof of Theorem \ref{thm: main} (b)]

Let $p$ be a prime, and let
$K=\mathbb{F}_p(x_1,\ldots ,x_{p^2-1})$.  We let $k=\mathbb{F}_p(x_1^p,\ldots, x_{p^2-1}^p)$ and we let $\delta$ be the $k$-linear derivation of $K$ given by
$\delta(x_i)=x_{i+1}$ for $i=1,\ldots ,p^2-1$, where we take $x_{p^2}=x_1$.  Since $k$ has characteristic $p>0$, we have $\delta^{p^i}$ is a $k$-linear derivation for every $i\ge 0$, and since $\delta^{p^2}(x_i)=\delta(x_i)=x_{i+1}$ for $i=1,\ldots ,p^2-1$, $\delta^{p^{j+2}}=\delta^{p^j}$ for every $j\ge 0$.  We let $\delta':=\delta^p$, which as we have just remarked is a $k$-linear derivation of $K$. 
We let $A=K[x;\delta]$ and we let $B=K[x';\delta']$.  Since ${\rm ad}_u^p = {\rm ad}_{u^p}$ for $u$ in a ring of characteristic $p$, we have $z:=x^{p^2}-x$ and $z':=(x')^{p^2}-x'$ are central by the above remarks.  We claim that $A$ and $B$ have Gelfand-Kirillov dimension one, $A\not\cong B$, and $A[t]\cong B[t']$.

Since $[K:k]<\infty$, $A$ and $B$ are finitely generated $k$-algebras of Gelfand-Kirillov dimension one \cite[Proposition 3.5]{KL}.
We construct an isomorphism $\Phi : A[t]\to B[t']$ as follows.  
We define $\Phi(\alpha)=\alpha$ for $\alpha\in K$, $\Phi(x)=(x')^p + t'$ and $\Phi(t)=(x')^{p^2}-x' + (t')^p$.  Then to show that $\Phi$ extends to a $k$-algebra homomorphism from $A[t]$ to $B[t']$, it suffices to show that $$\delta(\alpha)=\Phi([x,\alpha]) = [\Phi(x),\alpha]$$ for $\alpha\in K$ and that $\Phi(t)$ is central.  For $\alpha\in K$, $$[\Phi(t), \alpha]=[(x')^{p^2} - (x'),\alpha] = (\delta')^{p^2}(\alpha) - \delta'(\alpha)=0$$ and since $\Phi(t)$ also commutes with $x'$, $\Phi(t)$ is central.  To show that 
$$\delta(\alpha)=\Phi([x,\alpha]) = [\Phi(x),\alpha]$$ for $\alpha\in K$, observe that $\Phi([x,\alpha])= \Phi(\delta(\alpha))=\delta(\alpha)$ and $$[\Phi(x),\Phi(\alpha)] = [(x')^p+t',\alpha] = (\delta')^p(\alpha) =\delta^{p^2}(\alpha) = \delta(\alpha).$$  Thus $\Phi$ induces a homomorphism from $A[t]$ to $B[t']$. 
We claim that $\Phi$ is onto.
We have $\Phi(z) =  (z')^p + (t')^{p^2} - t'$ and $\Phi(t)=z'+(t')^p$.  
In particular, $$\Phi(z-t^p) = (z')^p + (t')^{p^2} -t' - (z')^p - (t')^{p^2} = -t'$$ and so
$\Phi(t+(z-t^p)^p) = z'$.  Thus $K$, $t'$ and $z'$ are in the image of $\Phi$.  Since $\Phi(x)=(x')^p +t'$ we also have $(x')^p\in {\rm Im}(\Phi)$. Finally, observe that
$z'=(x')^{p^2}-x'$ and since $z'$ and $(x')^p$ are in the image of $\Phi$, so is $$x' =(x')^{p^2}-z'=((x')^{p})^p-z'.$$  Thus $x', z'$ and $K$ are in the image of $\Phi$ and so $\Phi$ is onto.  Let $I$ denote the kernel of $\Phi$.  Then since $\Phi:A[t]\to B[t]$ is onto, we have $A[t]/I\cong B[t]$.  But $A[t]$ and $B[t]$ are both affine domains of Gelfand-Kirillov dimension two, and so $I$ is necessarily zero \cite[Proposition 3.15]{KL}.  Thus $\Phi$ is an isomorphism and so $A[t]\cong B[t]$. 

Thus it only remains to show that $A\not\cong B$ as $k$-algebras. 
To see this, suppose that $\Psi: A\to B$ is a $k$-algebra isomorphism.  Then since the units group of $A$ and $B$ are both $K^*$, $\Psi$ induces a $k$-algebra automorphism of $K$; furthermore, every $\alpha\in K$ satisfies $\alpha^p \in k$ and for $\beta\in k$ there is a unique $\alpha\in K$ such that $\alpha^p=\beta$.  Since $\Psi$ is the identity on $k$, $\Psi$ is the identity on $K$. Thus $\Psi(x) = p(x')$ for some $p(x')\in K[x';\delta']\setminus K$. Let $d\ge 1$ denote the degree of $p(x')$ as a polynomial in $x'$.  If $d>1$, it is straightforward to show that $\Psi$ cannot be onto, as every element in the image of $\Psi$ necessarily then has degree in $x'$ equal to a multiple of $d$.  Since $\Psi(x)\not\in K$, we see $\Psi(x)=\alpha x' + \beta$ with $\alpha\in K^*$ and $\beta\in K$.  Since $\Psi$ is an isomorphism, for $\zeta\in K$ we have 
$$\delta(\zeta)=\Psi(\delta(\zeta))=\Psi([x,\zeta]) =[\Psi(x),\Psi(\zeta)]= [\alpha x'+\beta, \zeta]=\alpha[x',\zeta]=\alpha\delta'(\zeta).$$
But by construction $\delta(x_1)=x_2$ and $\delta'(x_1)=x_{p+1}$ and so $\alpha = x_2/x_{p+1}$.
We also have $\delta(x_2)=x_3$ and $\delta'(x_2)=x_{p+2}$, and so $\alpha=x_3/x_{p+2}$, which gives $x_2 x_{p+2}= x_3 x_{p+1}$, where we take $x_{p+2}=x_1$ when $p=2$.  This is a contradiction.  Thus $A\not\cong B$.
\end{proof}

A key feature of these examples is that they have centres that are not inseparably closed.  It is natural to ask whether affine domain $A$ of GK dimension one are cancellative when one adds the assumption that the base field is inseparably closed. 
\begin{question}
Let $k$ be a field of positive characteristic and let $A$ be an affine domain $A$ of GK dimension one with the property that $k$ is inseparably closed in $A$. Is $A$ cancellative? 
\end{question}
If this question has a negative answer, a counterexample must be very constrained.  By work of Lezama, Wang, and Zhang \cite{LWZ}, if $A$ is a counterexample we have $Z(A)\cong k'[x]$ for some finite extension $k'$ of $k$, furthermore, $A$ is Azumaya and the Brauer group of $k'[x]$ cannot be trivial.  By Proposition \ref{pure}, we have $\ML^I(A)=A$ and yet we must also have $\ML^{H'}(A)\neq A$ by Proposition \ref{chrigid}.

\section{Skew Cancellativity}\label{S5}

We now consider the case of when an isomorphism of skew polynomial extensions $R[x;\sigma,\delta]\cong S[x;\sigma';\delta']$ implies that $R$ and $S$ are isomorphic.  We consider the case when $R$ and $S$ are finitely generated commutative integral domains of Krull dimension one over a field.  We observe that when $\sigma,\sigma'$ are the identity maps and $\delta,\delta'$ are zero, the question reduces to the classical cancellation problem for affine curves, answered by Abhyankar, Eakin, and Heinzer \cite{AEH}.  To prove Theorem \ref{thm: main2}, we must consider two types of extensions: skew extensions of automorphism type and skew extensions of derivation type.  We first look at the automorphism type case, in which the analysis is more straightforward.  
\begin{lemma} Let $k$ be a field, let $R$ be an affine commutative domain over $k$ of Krull dimension one, and let $\sigma$ be a $k$-algebra automorphism of $R$ that is not the identity. If $A$ is a commutative domain of Krull dimension one that is a homomorphic image of $R[x;\sigma]$ then either $A\cong R$ or $A\cong K[x]$ for some finite extension $K$ of $k$; moreover $R$ occurs as a homomorphic image of $R[x;\sigma]$.
\label{lem:sigma}
\end{lemma}
\begin{proof}
We consider prime commutative homomorphic images of $T:=R[x;\sigma]$ of Krull dimension one.  Observe that if $P$ is a prime ideal of $T$ such that $T/P$ is commutative, then since $T/P$ is an integral domain and $R/(P\cap R)$ embeds in $T/P$, $R/(P\cap R)$ is also an integral domain.  Since $R$ is an integral domain of Krull dimension one, either $P\cap R=(0)$ or $P\cap R=I$, with $I$ a maximal ideal of $R$.  In the former case, observe that since $xr = \sigma(r) x \equiv x \sigma(r)~(\bmod\, P)$, we have $x(r-\sigma(r))\in P$.  Moreover, since $\sigma$ is not the identity and $P$ is completely prime, we necessarily have $x\in P$.  Thus $T/P$ is a homomorphic image of $R[x;\sigma]/(x)\cong R$.  Since $T/P$ and $R$ are both integral domains of Krull dimension one, we then have $T/P\cong R$ in this case.  
In the case where $P\cap R=I$, with $I$ a maximal ideal of $R$.  We claim that $I=I^{\sigma}$.  To see this, suppose that this is not the case.  Then since $I$ is maximal, $I+\sigma(I)=R$.  In particular, there are $a,b\in I$ such that $a+\sigma(b)=1$.  Then $ax, xb\in P$ and so $ax+xb\in P$.  But $ax+xb= (a+\sigma(b))x  = x$ and so $x\in P$.  Thus $T/P$ is a homomorphic image of $R/I$, which contradicts the assumption that $T/P$ has Krull dimension one.  Hence $I=\sigma(I)$. Then by the Nullstellensatz $K:=R/I$ is a finite extension of $k$ and $\sigma$ induces a $k$-algebra automorphism of $K$.  We next claim that $\sigma$ is the identity on $K$; if not, there is some $\lambda\in K$ such that $\lambda \not\equiv \sigma(\lambda) ~(\bmod\, P)$.  But since $[\lambda,x]=(\lambda-\sigma(\lambda))x\in P$ and since $P$ is completely prime, we again have $x\in P$, which gives $T/P\cong K$, a contradiction.  Thus $\sigma$ induces the identity map on $R/I=K$ and so $T/IT \cong K[x]$.  Since $P$ contains $IT$, we then see that $T/P$ is a homomorphic image of $K[x]$ and since $T/P$ has Krull dimension one, we have $T/P\cong K[x]$.  The result follows.
\end{proof}
\begin{proposition}  Let $k$ be a field and let $R$ be an affine commutative domain over $k$ of Krull dimension one.  If $R[x;\sigma]\cong S[x;\sigma']$ then $R\cong S$.
\label{sig}
\end{proposition}
\begin{proof}
If $\sigma$ is the identity then both $R[x;\sigma]$ and $S[x;\sigma']$ are commutative and so $\sigma'$ is also the identity and the result follows from \cite{AEH}. Hence we may assume that $\sigma$ and $\sigma'$ are not the identity maps on their respective domains.
By Lemma \ref{lem:sigma}, the set of isomorphism classes of prime commutative images of $R[x;\sigma]$ of Krull dimension one is contained in $\{K[x] \colon [K:k]<\infty\}\cup R$, with $R$ occurring on the list.  
Similarly, the set of isomorphism classes of prime commutative images of $S[x;\sigma']$ of Krull dimension one is contained in $\{K[x] \colon [K:k]<\infty\}\cup S$, with $S$ occurring on the list.  It follows that either $R\cong S$ or $R\cong K[x]$ for some finite extension $K$ of $k$.  Similarly, either $S\cong R$ or $S\cong K'[x]$ for some finite extension $K'$ of $k$.  Thus we may assume without loss of generality that $R= K[t]$ and $S\cong K'[t]$ with $K,K'$ finite extensions of $k$.  Then the $k$-algebra isomorphism $R[x;\sigma]\to S[x;\sigma']$ restricts to an isomorphism of the units groups.
Since the units groups of $R[x;\sigma]=K^*$ and the units group of $S[x;\sigma']$ is $(K')^*$, we see the isomorphism restricts to a $k$-algebra isomorphism between $K$ and $K'$.  Thus $K\cong K'$ and so $R\cong S$.
\end{proof}
We now prove a lemma, which is a straightforward extension of earlier work (see \cite[Lemma 21]{Ma1}, \cite[Lemma 3.5]{BZ1}).
 \begin{lemma} Let $k$ be a field of characteristic zero and let $A$ be a finitely generated Ore domain over $k$. If $\ML(A)= A$ then $\ML(A[x;\delta])=\ML(A)$.  
 \label{lem:delta1}
 \end{lemma}
\begin{proof} Let $\mu$ be a locally nilpotent derivation of $A$.  Then $\mu$ extends to a locally nilpotent derivation of $A[x;\delta]$ by declaring that $\mu(x)=1$.  Then the kernel of this extension of $\mu$ is equal to ${\rm ker}(\mu|_A)$ and hence 
 $\ML(A[x;\delta])\subseteq \ML(A)$.  Now we show that the reverse containment holds. Let $\mu$ be a locally nilpotent derivation of $A[x;\delta]$ and suppose that $\mu$ is not identically zero on $\ML(A)$. Since $A$ is finitely generated there is some largest $m\ge 0$ such that for $r\in A$ we have
 $$\mu(r) = \partial(r) x^m + {\rm lower ~degree~terms},$$ with $\partial$ a derivation of $A$ that is not identically zero on $\ML(A)$. If $m=0$ then $\partial$ is a locally nilpotent derivation of $A$ and hence vanishes on $\ML(A)$, a contradiction.  Thus we may assume that $m>0$.  We now argue as in the three cases given in \cite[Lemma 3.5]{BZ1}.
 \end{proof}
 The following result is due to Crachiola and Makar-Limanov \cite[Lemma 2.3]{CM}.
 \begin{remark}
 Let $k$ be a field of characteristic zero and let $R$ be an affine commutative domain of Krull dimension one. Then either $\ML(R)=R$ or $R\cong k'[t]$ for some finite field extension $k'$ of $k$.  
 \end{remark}
 \begin{proof} Suppose that $\ML(R)\neq R$.  Then there is a locally nilpotent derivation $\delta$ of $R$ that is not identically zero on $R$.  In particular, the kernel of $\delta$ is a subalgebra $R_0$ of $R$.  By Remark \ref{rem:direct}, $R_0$ is finite-dimensional and hence a finite extension $k'$ of $k$.  Then pick $x\in R$ such that $\delta(x)\neq 0$ and $\delta^2(x)=0$.  Then $\delta(x)\in (k')^*$ and so we may rescale and assume that $\delta(x)=1$.  Then by Lemma \ref{gsliceth}, $R\cong k'[x]$, as required.  
 \end{proof}
 \begin{corollary} Let $k$ be a field of characteristic zero, let $R$ be a finitely generated $k$-algebra that is a commutative domain of Krull dimension one, and let $\delta$ be a $k$-linear derivation of $R$.  Then either $R\cong k'[t]$ for some finite extension $k'$ of $k$ or $\ML(R[x;\delta])=R$.
 \end{corollary}
 \begin{proposition} Let $k$ be a field of characteristic zero and let $R$ and $S$ be affine commutative domains over $k$ of Krull dimension one.  If $\delta$ and $\delta'$ are respectively $k$-linear derivations of $R$ and $S$ and $R[x;\delta]\cong S[x;\delta']$, then $R\cong S$.
 \label{del}
 \end{proposition}
 \begin{proof}
 By Lemma \ref{lem:delta1}, we have $\ML(R)=\ML(R[x;\delta]) \cong \ML(S[x;\delta']) = \ML(S)$.  Now if neither $R$ nor $S$ is isomorphic to $k'[x]$, with $k'$ some finite extension of $k$, then $R=\ML(R)\cong \ML(S)=S$ and we get the result.  If $R$ is isomorphic to $k'[x]$ for some finite extension of $k$ and $S$ is not isomorphic to an algebra of this type, then $k'=\ML(R)\cong \ML(S)=S$, which is impossible.  Thus we may assume that $R\cong k'[t]$ and $S\cong k''[t]$ where $k'$ and $k''$ are finite extensions of $k$.  But now the units group of $R[x;\delta]$ is $(k')^*$ and the units group of $S[x;\delta']$ is $(k'')^*$ and so the isomorphism from $R[x;\delta]\to S[x;\delta']$ restricts to an isomorphism between $k'$ and $k''$ and so $R\cong S$ in this case.
 \end{proof}
 We do not know whether Proposition \ref{del} is true when the base field $k$ has positive characteristic.  
We compare the examples from Theorem \ref{thm: main} (b) with the positive characteristic case of Proposition \ref{del}.  In positive characteristic, there exists a field $k$ and a finite extension $K$ of $k$ and $k$-linear derivations $\delta,\delta'$ of $K$ such that $K[t;\delta][x]\cong K[t;\delta'][x]$ but $K[t;\delta]\not\cong K[t;\delta']$.  But we can extend $\delta$ and $\delta'$ to $K[x]$ by declaring that $\delta(x)=\delta'(x)=0$ and we have
$$K[t;\delta][x]\cong K[x][t;\delta]\cong K[x][t;\delta']\cong K[t;\delta'][x].$$
So the algebra $K[t;\delta][x]\cong K[x][t;\delta]$ is cancellative with respect to the variable $t$ but not with respect to the variable $x$.  Thus these examples do not give rise to counterexamples to the positive characteristic of Proposition \ref{del}.
\begin{proof}[Proof of Theorem \ref{thm: main2}] This follows immediately from Propositions \ref{sig} and \ref{del}.
 \end{proof}
We do not know whether cancellation holds for skew polynomial extensions of mixed type with coefficient rings being domains of Krull dimension one.  We finish by posing the following unresolved question, which---if the answer were affirmative---would unify the two cases in Theorem \ref{thm: main2} and would also extend Proposition \ref{del} to base fields of positive characteristic.
\begin{question}
Let $k$ be a field, let $R$ be an affine commutative domain over $k$ of Krull dimension one, and let $\sigma$ and $\delta$ be respectively a $k$-algebra automorphism and a $k$-linear $\sigma$-derivation of $R$.  Is $R[x;\sigma,\delta]$ cancellative?
\end{question}


\end{document}